\documentclass[conference]{IEEEtran}
\IEEEoverridecommandlockouts

\usepackage{amsmath}
\usepackage{amsthm}
\usepackage{amssymb}
\usepackage{amsfonts}
\usepackage{mathtools}
\usepackage{mathrsfs}
\usepackage{graphics} 
\usepackage{graphicx}
\usepackage{textcomp}
\usepackage{color}
\usepackage{multirow}
\usepackage{array}
\usepackage{xcolor}
\usepackage{epsfig} 
\usepackage{epstopdf}
\usepackage[noadjust]{cite}
\usepackage{tikz}
\usetikzlibrary{calc,positioning,shapes,shadows,arrows,fit}
\usepackage{float}
\usepackage[font = footnotesize]{caption}

\newcolumntype{C}[1]{>{\centering\let\newline\\\arraybackslash}m{#1}}

\usepackage{fancyref}
\usepackage{hyperref}

\newtheorem{theorem}{Theorem}
\newtheorem{lemma}{Lemma}

\theoremstyle{definition}
\newtheorem{definition}{Definition}
\newtheorem{remark}{Remark}

\newtheorem{problem}{Problem}

\begin{document}

\title{A Tube-based MPC Scheme for Interaction Control of Underwater Vehicle Manipulator Systems\\
\thanks{This work was supported by the H2020 ERC Grant BUCOPHSYS, the EU H2020 Co4Robots project, the Swedish Foundation for Strategic Research (SSF), the Swedish Research Council (VR) and the Knut och Alice Wallenberg Foundation (KAW).}
}

\author{\IEEEauthorblockN{Alexandros Nikou, Christos K. Verginis and Dimos V. Dimarogonas}
\IEEEauthorblockA{Department of Automatic Control \\
School of Electrical Engineering and Computer Science\\
KTH Royal Institute of Technology, Stockholm, Sweden \\
{\tt \{anikou,cverginis,dimos\}@kth.se}}
}

\maketitle

\begin{abstract}
Over the last years, the development of Autonomous Underwater Vehicles (AUV) with attached robotic manipulators, the so-called Underwater Vehicle Manipulator System (UVMS), has gained significant research attention, due to the ability of interaction with underwater environments. In such applications, force/torque controllers which guarantee that the end-effector of the UVMS applies desired forces/torques towards the environment, should be designed in a way that state and input constraints are taken into consideration. Furthermore, due to their complicated structure, unmodeled dynamics as well as external disturbances may arise. Motivated by this, we proposed a robust Model Predicted Control Methodology (NMPC) methodology which can handle the aforementioned constraints in an efficient way and it guarantees that the end-effector is exerting the desired forces/torques towards the environment. Simulation results verify the validity of the proposed framework.
\end{abstract}

\section{Introduction}
Most of the underwater manipulation tasks, such as maintenance of ships, underwater weld inspection, surveying oil/gas searching, require the manipulator mounted on the vehicle to be in contact with the underwater object or environment (see \cite{antonelli, cieslak2015autonomous}). The aforementioned tasks are usually complex due to highly nonlinear dynamics, the presence of uncertainties, external disturbances as well as state and control input (actuation) constraints. Thus, these constraints should be taken into account in the force control design process in an efficient way.

Motivated by the aforementioned, this paper considers the modeling of a general UVMS in compliant contact with a planar surface, and the development of  a constrained Nonlinear Model Predictive Control (NMPC) scheme for force/torque control. NMPC for manipulation of nominal system dynamics has been proposed in \cite{alex_med} for stabilization of ground vehicles with attached manipulators to pre-defined positions. In this work, we propose a novel robust tube-based NMPC force control approach that efficiently deals with state and input constraints and achieves a desired exerted force from the UVMS to the environment.  In particular, the controller consists of two terms: a nominal control input, which is computed on-line and is the outcome of a Finite Horizon Optimal Control Problem (FHOCP) that is repeatedly solved at every sampling time, for its nominal system dynamics; and an additive state feedback law which is computed off-line and guarantees that the real trajectory of the closed-loop system will belong to a hyper-tube centered along the nominal trajectory. The volume of the hyper-tube depends on the upper bound of the disturbances, the bounds of the Jacobian matrix as well as Lipschitz constants of the UVMS dynamics. Under the assumption that the FHOCP is feasible at time $t = 0$, we guarantee the boundedness of the closed-loop system states.

The rest of this manuscript is structured as follows: Section \ref{sec:notation_preliminaries} provides the notation that will be used as well as necessary background knowledge; in Section \ref{sec:problem_formulation}, the problem treated in this paper is formally defined; Section \ref{sec:main_results} contains the main results of the paper; Section \ref{sec:simulation_results} is devoted to numerical simulations; and in Section \ref{sec:conclusions}, conclusions and future research directions are discussed.

\section{Notation and Preliminaries} \label{sec:notation_preliminaries}

Define by $\mathbb{N}$ and $\mathbb{R}$ the sets of positive integers and real numbers, respectively. Given the set $\mathcal{S}$, define by $S^n \coloneqq S \times \dots \times S$, its $n$-fold Cartesian product. Given vector $z \in \mathbb{R}^{n}$ define by $$\|z\|_{2} \coloneqq \sqrt{z^\top z}, \ \ \|z\|_{P} \coloneqq \sqrt{z^\top P z},$$ its Euclidean and weighted norm, with $P \ge 0$. Given vectors $z_1$, $z_2 \in \mathbb{R}^3$, $\mathcal{S}: \mathbb{R}^3 \to \mathfrak{so}(3)$ stands for the skew-symmetric matrix defined according to $\mathcal{S}(z_1) z_2 = z_1 \times z_2$ where $$\mathfrak{so}(3) \coloneqq \left\{\mathcal{S} \in \mathbb{R}^{3\times 3} : z^\top \mathcal{S}(\cdot) z = 0, \forall z \in \mathbb{R}^{3} \right\}.$$  $\lambda_{\scriptscriptstyle \min}(P)$ stands for the minimum absolute value of the real part of the eigenvalues of $P \in \mathbb{R}^{n \times n}$; $0_{m \times n} \in \mathbb{R}^{m \times n}$ and $I_n \in \mathbb{R}^{n \times n}$ stand for the $m \times n$ matrix with all entries zeros and the identity matrix, respectively. Given coordination frames $\Sigma_i$, $\Sigma_j$, denote by $R^j_i$ the transformation from $\Sigma_i$ to $\Sigma_j$. Given~sets~$\mathcal{S}_1$, $\mathcal{S}_2$~$\subseteq \mathbb{R}^n$, $\mathcal{S} \subseteq \mathbb{R}^{m}$~and~matrix $B \in \mathbb{R}^{n \times m}$,~the \emph{Minkowski addition}, the~\emph{Pontryagin~difference} and the \emph{matrix-set multiplication} are respectively defined by: 
\begin{align*}
\mathcal{S}_1 \oplus \mathcal{S}_2 & \coloneqq \{s_1 + s_2 : s_1 \in \mathcal{S}_1, s_2 \in \mathcal{S}_2\}, \\ 
\mathcal{S}_1 \ominus \mathcal{S}_2 & \coloneqq \{s_1 : s_1+s_2 \in \mathcal{S}_1, \forall s_2 \in \mathcal{S}_2\}, \\
B \circ \mathcal{S} & \coloneqq \{b:  b = Bs, s \in \mathcal{S} \}.
\end{align*}

\begin{lemma} \cite{alex_IJRNC_2018} \label{lemma:basic_ineq}
For any constant $\rho > 0$, vectors $z_1$, $z_2 \in \mathbb{R}^n$ and matrix $P \in \mathbb{R}^{n \times n}$, $P > 0$ it holds that $$z_1 P z_2 \le \tfrac{1}{4 \rho} z_1^\top P z_1 + \rho z_2^\top P z_2.$$
\end{lemma}

\begin{definition} \label{def:RPI_set} \cite{alex_IJRNC_2018}
Consider a dynamical system $\dot{\chi} = f(\chi,u,d)$ where: $\chi \in \mathcal{X}$, $u \in \mathcal{U}$, $d \in \mathcal{D}$ with initial condition $\chi(0) \in \mathcal{X}$. A set $\mathcal{X}' \subseteq \mathcal{X}$ is a \emph{Robust Control Invariant (RCI) set} for the system, if there exists a feedback control law $u \coloneqq \kappa(\chi) \in \mathcal{U}$, such that for all $\chi(0) \in \mathcal{X}'$ and for all $d \in \mathcal{D}$ it holds that $\chi(t) \in \mathcal{X}'$ for all $t \ge 0$, along every solution $\chi(t)$.
\end{definition}

\section{Problem Formulation} \label{sec:problem_formulation}

\subsection{Kinematic Model}

Consider a UVMS which is composed of an AUV and a $n$ Degree Of Freedom (DoF) manipulator mounted on the base of the vehicle. The AUV can be considered as a six DoF rigid body with position and orientation vector $\eta \coloneqq [x, y, z~\vline~\phi, \theta, \psi]^\top \in \mathbb{R}^6$, where the components of the vectors have been named according to SNAME  \cite{SNAME} as surge, sway, heave, roll, pitch and yaw respectively. The joint angular position state vector of the manipulator is defined by $q \coloneqq [ q_1,\dots,q_n]^\top \in \mathbb{R}^n$. Define by $\dot{q} \coloneqq [\dot{q}_1,\dots,\dot{q}_n]^\top \in \mathbb{R}^n$ the corresponding joint velocities.

In order to describe the motion of the combined system, the earth-fixed inertial frame $\Sigma_I$, the body-fixed frame $\Sigma_B$ and the end-effector fixed frame $\Sigma_E$ are introduced (see Fig. \ref{fig:uvms_frames}). Moreover, without loss of generality, the reference frame $\Sigma_0$ is chosen to be located at the manipulator's base, and the frames $\Sigma_1, \ldots, \Sigma_n$ are located to the $1$-st$,\ldots,n$-th link of the manipulator, respectively, under the Denavit-Hartenberg convention \cite{sciavicco2012modelling}. The translational and rotational kinematic equations for the AUV system (see \cite{antonelli}) are given by:
\begin{subequations} \label{eq:kin}
\begin{align}
	\dot{\eta}
	& =
	\begin{bmatrix}
	\dot{\eta}_1 \\
	\dot{\eta}_2 \\
	\end{bmatrix}
	= \mathfrak{J}(\eta_2)
	\begin{bmatrix}
	\nu_1 \\
	\nu_2 \\
	\end{bmatrix}, \\
	\mathfrak{J}(\eta_2) & \coloneqq 
	\begin{bmatrix}
	\mathfrak{J}_1(\eta_2) & 0_{3 \times 3} \\
	0_{3 \times 3} & \mathfrak{J}_{2}(\eta_2) \\
	\end{bmatrix}, \\
	\mathfrak{J}_1(\eta_2) & \coloneqq
	\begin{bmatrix}
	c_{\theta} c_{\psi} & s_{\phi} s_{\theta} c_{\psi}-s_{\psi} c_{\phi} &
	s_{\theta} c_{\phi} c_{\psi}+s_{\phi} s_{\psi} \\
	s_{\psi} c_{\theta} & s_{\phi} s_{\theta} s_{\psi}+c_{\phi} c_{\psi} &
	s_{\theta} s_{\psi} c_{\phi}-s_{\phi} c_{\psi} \\
	-s_{\theta} & s_{\phi} c_{\theta} & c_{\phi} c_{\theta} \\
	\end{bmatrix}, \\
	\mathfrak{J}_2(\eta_2) & \coloneqq
	\begin{bmatrix}
	1 & \tfrac{s_{\phi} s_{\theta}}{c_{\theta}} & \tfrac{c_{\phi} s_{\theta}}{c_{\theta}} \\
	0 & c_{\phi} & -s_{\phi} \\
	0 & \tfrac{s_{\phi}}{c_{\theta}} & \tfrac{c_{\phi}}{c_{\theta}} \\
	\end{bmatrix},
\end{align}
\end{subequations}
where $\eta_{1} \coloneqq \left[ x, y, z \right]^{\tau} \in \mathbb{R}^3$, $\eta_{2} \coloneqq \left[\phi, \theta, \psi \right]^\top \in \mathbb{R}^3 $ denote the position vector and the orientation vector of the frame $\Sigma_B$ relative to the frame $\Sigma_I$, respectively; $ \nu_{1}$, $\nu_{2} \in \mathbb{R}^3 $ denote the linear and the angular velocity of the frame $\Sigma_B$ with respect to $\Sigma_I$ respectively; $\mathfrak{J}(\eta_2) \in \mathbb{R}^{6 \times 6}$ stands for the Jacobian matrix transforming the velocities from $\Sigma_B$ to $\Sigma_I$; $\mathfrak{J}_{1}(\eta_2)$, $\mathfrak{J}_2(\eta_2) \in \mathbb{R}^{3 \times 3}$ are the corresponding parts of the Jacobian related to position and orientation, respectively; The notation $s_{\varsigma}$ and $c_{\varsigma}$ stand for the trigonometric functions $\sin(\varsigma)$ and $\cos(\varsigma)$ of an angle $\varsigma \in \mathbb{R}$, respectively.

\begin{figure}[t!]
\centering
\includegraphics[scale=0.45]{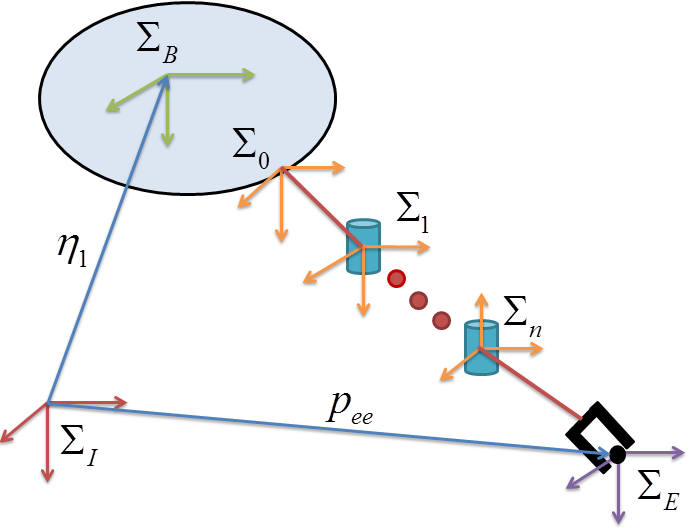}
\caption{An AUV Equipped with a n DoF manipulator}
\label{fig:uvms_frames}
\end{figure}

\noindent Denote by $$\mathfrak{q} \coloneqq \left[\eta_1^\top, \eta_2^\top, q^\top \right]^\top\in\mathbb{R}^{6+n},$$ the pose configuration vector of the UVMS. Let $\mathfrak{p}$, $\mathfrak{o} \in \mathbb{R}^3$ be the position and orientation vectors of the end-effector with reference to the frame $\Sigma_I$, respectively. The vectors $\mathfrak{p}$, $\mathfrak{o}$ depend on the pose $\mathfrak{q}$ and they can be obtained by the the following homogeneous transformation:
\begin{equation}
\mathfrak{T}(\mathfrak{q}) \coloneqq
\begin{bmatrix}
R_E^{I}(\mathfrak{q}) & \mathfrak{p}(\mathfrak{q}) \\
0_{1 \times 3} & 1 \\
\end{bmatrix}
=
T_B^I T_0^B T_1^0 \cdots T_n^{n-1} T_E^n, \label{eq:forw_kinematics}
\end{equation}
where: $T^j_i$ is the homogeneous transformation matrix describing the position and orientation of frame $\Sigma_i$ with reference to the frame $\Sigma_j$ with $i$, $j \in \{1,\dots, n, I, 0, B, E\}$. The end-effector linear velocity $\dot{\mathfrak{p}} \in \mathbb{R}^3$ and the time derivative or Euler angles $\dot{\mathfrak{o}} \in \mathbb{R}^3$ are related to the body-fixed velocities $\nu_1$, $\nu_2$ and $\dot{q}$ with the following \emph{kinematics model:}
\begin{equation} \label{eq:kinematics}
\dot{\chi}
=
J(\mathfrak{q}) \zeta, 
\end{equation}
where $$\chi \coloneqq [\mathfrak{p}^\top, \mathfrak{o}^\top]^\top \in \mathbb{R}^{6}, \ \ \zeta \coloneqq \left[ \nu_1^\top, \nu_2^\top, \dot{q}^\top \right]^\top \in \mathbb{R}^{{6+n}},$$ is the body-fixed system velocity vector. The Jacobian transformations matrices $$J(\mathfrak{q}) \in \mathbb{R}^{6 \times (6+n)}, \ \ J_{\rm pos}(\mathfrak{q}) \in \mathbb{R}^{3 \times (6 +n)}, \ \ J_{\rm{or}}(\mathfrak{q}) \in \mathbb{R}^{3 \times (6 +n)},$$ are respectively defined by:
\begin{align*} 
J(\mathfrak{q})  &\coloneqq 	\begin{bmatrix}
J_{\rm pos}(\mathfrak{q}) \\
J_{\rm or}(\mathfrak{q}) \\
\end{bmatrix}, \\
J_{\rm pos}(\mathfrak{q}) &\coloneqq 
\left[\mathfrak{J}_1(\eta_2)~\vline~- \mathfrak{J}_1(\eta_2) \mathcal{S}(p_{ee})~\vline~R_0^{I} J_{e,1} \right], \\
J_{\rm or}(\mathfrak{q}) &\coloneqq 
\left[0_{3 \times 3}~\vline~\mathfrak{J}_2(\mathfrak{o}) R_B^E~\vline~\mathfrak{J}_2(\mathfrak{o}) R_0^{E} J_{e,2} \right].
\end{align*}
In the latter, the vector $p_{ee} \in \mathbb{R}^{3}$ is the local position of the end-effector with reference to the frame $\Sigma_B$; the matrices  $J_{e,1}$, $J_{e,2} \in \mathbb{R}^{3 \times n}$ represent the manipulator Jacobian matrices with respect to the frame $\Sigma_0$; and $\mathcal{S}(\cdot)$ the skew-symmetric matrix as given in Section \ref{sec:notation_preliminaries}. For the aforementioned transformations we refer to \cite{sciavicco2012modelling}.

\subsection{Dynamic Model}
When the end-effector of the robotic system is in contact with the environment, the force at the tip of the manipulator acts on the whole system according to the following uncertain nonlinear dynamics:
\begin{align} \label{eq:dynamics}
\dot{\zeta} = f(\chi, \zeta)+ \mathfrak{u} + d(\mathfrak{q}, \zeta, t),
\end{align}
where:
\begin{align}
& \hspace{-4mm} f(\chi, \zeta) \coloneqq \notag \\
& \hspace{-4mm}  - M(\mathfrak{q})^{-1} \Big\{ C(\zeta, \mathfrak{q}) \zeta + D(\zeta, \mathfrak{q}) \zeta +g(\mathfrak{q}) + J^{\top}(\mathfrak{q}) \mathfrak{F}(\chi) \Big\}, \hspace{-4mm} \label{eq:func_h}
\end{align}
where $M(\mathfrak{q}) \in \mathbb{R}^{({6+n}) \times ({6+n})}$ is the inertia matrix for which it holds that: $z^\top M(\mathfrak{q}) z > 0$, $\forall z \in \mathbb{R}^{6+n}$; $C(\zeta, \mathfrak{q}) \in \mathbb{R}^{({6+n}) \times ({6+n})}$ is the matrix of Coriolis and centripetal terms; $D(\zeta, \mathfrak{q}) \in \mathbb{R}^{({6+n}) \times ({6+n})}$ is the matrix of dissipative effects; $d(\mathfrak{q}, \zeta, t) \in \mathbb{R}^{6+n}$ is a vector that models the external disturbances, uncertainties and unmodeled dynamics of the system; $g(\mathfrak{q}) \in \mathbb{R}^{({6+n})}$ is the vector of gravity and buoyancy effects; $\mathfrak{u} \in \mathbb{R}^{6+n}$ denotes the vector of the propulsion forces and moments acting on the vehicle in the frame $\Sigma_{B}$ as well as the joint torques; $\mathfrak{F}(\chi) \in \mathbb{R}^{6}$ is the vector of interaction forces and torques exerted by the end-effector towards the environment expressed in $\Sigma_I$.

In this paper, an interaction between the end-effector and a frictionless, elastically compliant surface is assumed. Then, according to \cite{siciliano_force_control}, the vector of interaction forces and torques that is exerted by the end-effector can be written as:
\begin{equation} \label{eq:force_F}
\mathfrak{F}(\chi) \coloneqq K (\chi - \chi_{\scriptscriptstyle \rm eq}),
\end{equation}
where $K \in \mathbb{R}^{6 \times 6}$, $K > 0$ stands for the stiffness matrix, which represents elastic coefficient of the environment, and $\chi_{\scriptscriptstyle \rm eq} \in \mathbb{R}^{6}$ is the given constant vector of the equilibrium position/orientation of the undeformed environment.

We also consider that the UVMS is in the presence of state and input constraints given by $\mathfrak{q} \in \mathcal{Q}$, $\zeta \in \mathcal{Z}$, $u \in \mathcal{U}$,
where $\mathcal{Q} \subseteq \mathbb{R}^{6+n}$, $\mathcal{Z} \subseteq \mathbb{R}^{6+n}$ and $\mathcal{U} \subseteq \mathbb{R}^{6+n}$ are \emph{connected sets containing the origin}. For certain technical reasons that will be presented thereafter, the constraints imposed to the configuration states $\mathfrak{q}$ are given by:
\begin{align} \label{eq:set_Q}
\mathcal{Q} \coloneqq \Big\{ \mathfrak{q} \in \mathbb{R}^{6+n} : & \ \lambda_{\scriptscriptstyle \min}\left[ \tfrac{J^{+}(\mathfrak{q})+J^{+}(\mathfrak{q})^\top}{2} \right] \ge \underline{J}, \notag \\
&  \ \ \|J(\mathfrak{q})\|_2 \le \overline{J}, \ \|\dot{J}(\mathfrak{q})\|_{2} \le \widetilde{J} \Big\},
\end{align}
where $J^{+}(\mathfrak{q}) \coloneqq J(\mathfrak{q}) J(\mathfrak{q})^\top$ and $\underline{J}$, $\overline{J}$, $\widetilde{J} > 0$. According to \eqref{eq:forw_kinematics}, the constraints $\mathfrak{q} \in \mathcal{Q}$ impose also constraints on the vector $\chi \in \mathcal{X} \subseteq \mathbb{R}^{6}$, where the set $\mathcal{X}$ can be computed by the transformation $\mathfrak{T}(\mathfrak{q})$, as given in \eqref{eq:forw_kinematics}. Note also that the function $f$ given in \eqref{eq:func_h} is continuously differentiable in the set $\mathcal{Q} \times \mathcal{X} \times \mathcal{Z}$. Furthermore, assume bounded disturbances $d \in \mathcal{D}$ where:
$\mathcal{D}$ $\coloneqq \big\{d \in \mathbb{R}^{6+n}:$ $\|d(\mathfrak{q}, \zeta, t)\|_{2} \le \widetilde{d}$, $\forall (\mathfrak{q}$, $\zeta) \in \mathcal{Q}$ $\times \mathcal{Z}\big\}$, where $\widetilde{d} > 0$. 

For the kinematics/dynamics \eqref{eq:kinematics},\eqref{eq:dynamics}, define the corresponding \emph{nominal kinematics/dynamics} by:
\begin{subequations}
\begin{align}
\dot{\overline{\chi}} & = J(\overline{\mathfrak{q}}) \overline{\zeta}, \label{eq:nom_kinematics} \\
 \dot{\overline{\zeta}} & = f(\overline{\chi}, \overline{\zeta})+ \overline{u}, \label{eq:nom_dynamics}
\end{align}
\end{subequations}
where $d(\cdot) \equiv 0$, $\overline{\mathfrak{q}} \in \mathcal{Q}$, $\overline{\chi} \in \mathcal{X}$, $\overline{\zeta} \in \mathcal{Z}$ and $\overline{u} \in \mathcal{U}$. Define the stack vector $\overline{\xi} \coloneqq [\overline{\chi}, \overline{\zeta}]^\top \in \mathbb{R}^{12+n}$ and consider the linear nominal system $\dot{\overline{\xi}} = A \overline{\xi} + B \overline{u}, \ \ A  \in \mathbb{R}^{(12+n) \times (12+n)}, \ \ B  \in \mathbb{R}^{(12+n) \times (6+n)}$,
which is the outcome of the Jacobian linearization of the nominal dynamics \eqref{eq:nom_kinematics},\eqref{eq:nom_dynamics} around the equilibrium point $\xi = 0$. Due to the dimension of the control input ($6+n > 6$), the stabilization of the state $\overline{\chi}$ to the desired state $\chi_{\scriptscriptstyle \rm des}$ can be achieved. Therefore, the linear system is stabilizable.

\subsection{Problem Statement}

\begin{problem} \label{problem}
Consider a UVMS composed of an AUV and an attached manipulator with $n$ DoF, which is in contact with a surface of a compliant environment. The UVMS is governed by the kinematics and dynamics models given in \eqref{eq:kinematics} and \eqref{eq:dynamics}, respectively. The system is in the presence of state and input constraints as well as bounded disturbances which are respectively given by:
\begin{align} \label{eq:constr}
\mathfrak{q} \in \mathcal{Q}, \ \chi \in \mathcal{X}, \ \zeta \in \mathcal{Z}, \ \mathfrak{u} \in \mathcal{U}, \ d \in \mathcal{D}.
\end{align}
Given a vector $\mathfrak{F}_{\scriptscriptstyle \rm des} \in \mathbb{R}^{6}$ that satisfies \eqref{eq:force_F} and stands for the desired force/torque vector  that the end-effector is required to exert towards a surface of the environment, design a \emph{feedback control law} $\mathfrak{u} \coloneqq \kappa(\chi, \zeta)$ such that $\lim\limits_{t \to \infty} \|\mathfrak{F}(\chi(t))-\mathfrak{F}_{\scriptscriptstyle \rm des}\|_{2} \to 0$, while all the constraints given in \eqref{eq:constr} are satisfied.
\end{problem}

\section{Main Results} \label{sec:main_results}

In this section, we propose a novel feedback control law that solves Problem \ref{problem} in a systematic way. Due to the fact that it is required to design a feedback control law that guarantees the minimization of the term $\|\mathfrak{F}(t)-\mathfrak{F}_{\scriptscriptstyle \rm des}\|_{2}$, as $t \to \infty$, under state and input constraints given by \eqref{eq:constr}, we utilize a Nonlinear Model Predictive Control (NMPC) framework \cite{michalska_1993, frank_1998_quasi_infinite, mayne_2000_nmpc}. Furthermore, since the UVMS is under the presence of disturbances/uncertainties $d \in \mathcal{D}$, we provide a robust analysis, the so-called tube-based robust NMPC approach \cite{yu_2013_tube, alex_IJRNC_2018}. In particular, first, the error states and the corresponding transformed constraints sets are defined in Section \ref{sec:error_constr}. Then, the proposed feedback control law consists of two parts: an on-line control law which is the outcome of a solution to a Finite Horizon Optimal Control Problem (FHOCP) for the nominal system dynamics (see Section \ref{sec:optimal_contol}); and a state feedback law which is designed off-line and guarantees that the real system trajectories always lie within a hyper-tube centered along the nominal trajectories (see \ref{sec:state_feedback_law}).

\subsection{Errors and Constraints} \label{sec:error_constr}

According to \eqref{eq:force_F}, for the error between the actual $\mathfrak{F}$ and the desired $\mathfrak{F}_{\scriptscriptstyle \rm des}$ forces/torques exerted from the end-effector to the surface it holds that: $\mathfrak{F}-\mathfrak{F}_{\scriptscriptstyle \rm des}$ $= K (\chi - \chi_{\scriptscriptstyle \rm eq})$ $-K (\chi_{\scriptscriptstyle \rm des}$ $- \chi_{\scriptscriptstyle \rm eq})$ $= K (\chi-\chi_{\scriptscriptstyle \rm des})$,
where $\chi_{\scriptscriptstyle \rm des} \coloneqq K^{-1} \mathfrak{F}_{\scriptscriptstyle \rm des} + \chi_{\scriptscriptstyle \rm eq} \in \mathbb{R}^{6}$. The latter implies that if we design a feedback control law $u = \kappa(\chi, \zeta)$ which guarantees that $\lim\limits_{t \to \infty} \|\chi(t)-\chi_{\scriptscriptstyle \rm des}\|_{2} \to 0$, while all the constraints given in \eqref{eq:constr} are satisfied, Problem \ref{problem} will have been solved.

Define the error state $e \coloneqq \chi-\chi_{\scriptscriptstyle \rm des} \in \mathbb{R}^{6}$. Then, the \emph{uncertain error kinematics/dynamics} are given by:
\begin{subequations}
\begin{align}
\dot{e} & = J(\mathfrak{q}) \zeta, \label{eq:unsrt_error_kin} \\
\dot{\zeta} & = f(e+\chi_{\scriptscriptstyle \rm des}, \zeta)+ \mathfrak{u} + d(\mathfrak{q}, \zeta, t), \label{eq:unsrt_error_dyn}
\end{align}
\end{subequations}
and the corresponding \emph{nominal error kinematics/dynamics} by:
\begin{subequations}
\begin{align}
\dot{\overline{e}} & = J(\overline{\mathfrak{q}}) \overline{\zeta}, \label{eq:nom_error_kin} \\
\dot{\overline{\zeta}} & = f(\overline{e}+\chi_{\scriptscriptstyle \rm des}, \overline{\zeta})+ \overline{u}, \label{eq:nom_error_dyn}
\end{align}
\end{subequations}
In order to translate the constraints for the state $\chi \in \mathcal{X}$ to constraints that are dictated regarding the error $e$, the constraints set $\mathcal{E} \coloneqq \{e \in \mathbb{R}^{6}: e \in \mathcal{X} \oplus (-\chi_{\scriptscriptstyle \rm des}) \}$ is introduced.

\subsection{Feedback Control Design} \label{sec:state_feedback_law}

\noindent Consider the feedback law:
\begin{equation} \label{eq:control_law_u}
\mathfrak{u} \coloneqq \overline{u}(\overline{e}, \overline{\zeta}) + \kappa(e, \zeta, \overline{e}, \overline{\zeta}),
\end{equation}
which consists of a nominal control law $\overline{u}(\overline{e}, \overline{\zeta}) \in \mathcal{U}$ and a state feedback law $\kappa(\cdot)$. The control action $\overline{u}(\overline{e}, \overline{\zeta})$ will be the outcome of a FHOCP for the nominal kinematics/dynamics \eqref{eq:nom_error_kin},\eqref{eq:nom_error_dyn} which is solved on-line at each sampling time. The state feedback law $\kappa(\cdot)$ is used to guarantee that the real trajectories $e(t)$, $\zeta(t)$, which are the solution to \eqref{eq:unsrt_error_kin},\eqref{eq:unsrt_error_dyn}, always remain within a bounded hyper-tube centered along the nominal trajectories $\overline{e}(t)$,~$\overline{\zeta}(t)$ which are~the~solution~ to~\eqref{eq:nom_error_kin},\eqref{eq:nom_error_dyn}.

Define by $\mathfrak{e} \coloneqq e - \overline{e} \in \mathbb{R}^{6}$ and $\mathfrak{z} \coloneqq \zeta - \overline{\zeta} \in \mathbb{R}^{6+n}$ the deviation between the real states of the uncertain system \eqref{eq:unsrt_error_kin},\eqref{eq:unsrt_error_dyn} and the states of the nominal system \eqref{eq:nom_error_kin},\eqref{eq:nom_error_dyn}, respectively, with $\mathfrak{e}(0) = \mathfrak{z}(0) = 0$. It will be proved hereafter that the trajectories $\mathfrak{e}(t)$, $\mathfrak{z}(t)$ remain invariant in compact sets. The dynamics of the states $\mathfrak{e}$, $\mathfrak{z}$ are written as:
\begin{subequations}
\begin{align}
\dot{\mathfrak{e}} &  = \mathfrak{b}(\chi, \overline{\chi}, \zeta) + J(\overline{\mathfrak{q}}) \mathfrak{z}, \label{eq:frak_e} \\
\dot{\mathfrak{z}} & = \mathfrak{l}(e, \overline{e}, \zeta, \overline{\zeta})+(\mathfrak{u}-\overline{u}) + d(\mathfrak{q}, \zeta, t), \label{eq:frak_z}
\end{align}
\end{subequations}
where the functions $\mathfrak{b}$, $\mathfrak{l}$ are defined by: $\mathfrak{b}(\chi, \overline{\chi}, \zeta) \coloneqq \mathfrak{c}(\chi, \zeta)-\mathfrak{c}(\overline{\chi}, \zeta)$, $\mathfrak{l}(e, \overline{e}, \zeta, \overline{\zeta}) \coloneqq  f(e+\chi_{\scriptscriptstyle \rm des}, \zeta)-f(\overline{e}+\chi_{\scriptscriptstyle \rm des}, \overline{\zeta})$, with $\mathfrak{c}(\chi, \zeta) \coloneqq J(\mathfrak{q}) \zeta$. Since the aforementioned functions are continuously differentiable, the following hold:
\begin{align*}
\|\mathfrak{b}(\cdot)\|_2 & = \|\mathfrak{c}(\chi, \zeta)-\mathfrak{c}(\overline{\chi}, \zeta)\|_2 \le L_{\scriptscriptstyle \mathfrak{c}} \|\chi-\overline{\chi}\|_2 = L_{\scriptscriptstyle \mathfrak{c}} \|\mathfrak{e}\|_2,  \\
\|\mathfrak{l}(\cdot)\|_{2} & \le \|f(e+\chi_{\scriptscriptstyle \rm des}, \zeta)- f(\overline{e}+\chi_{\scriptscriptstyle \rm des}, \zeta)\|_{2} \notag \\
&\hspace{12mm}+\|f(\overline{e}+\chi_{\scriptscriptstyle \rm des}, \zeta)- f(\overline{e}+\chi_{\scriptscriptstyle \rm des}, \overline{\zeta})\|_2 \notag \\
& \le L_1\|e-\overline{e}\|_{2} + L_2 \|\zeta-\overline{\zeta}\|_{2} \le L \left( \|\mathfrak{e}\|_{2} + \|\mathfrak{z}\|_{2} \right).
\end{align*}
The constant $L_{\scriptscriptstyle \mathfrak{c}}$ stands for the Lipschitz constant of function $\mathfrak{c}$ with respect to the variable $\chi$; $L_1$, $L_2$ stand for the Lipschitz constants of function $h$ with respect to the variables $\chi$ and $\zeta$, respectively, and $L \coloneqq \max\{L_1, L_2\}$.

\begin{lemma} \label{lemma:tube}
The state feedback law designed by:
\begin{equation} \label{eq:kappa_law}
\kappa(e, \overline{e}, \zeta, \overline{\zeta}) \coloneqq - k (e-\overline{e})-k \sigma J(\overline{q})^\top (\zeta -\overline{\zeta}),
\end{equation}
where $k$, $\sigma > 0$ are chosen such that the following hold:
\begin{subequations}
\begin{align}
\underline{\sigma} & > 0,\ \ \sigma \coloneqq \frac{L_{\scriptscriptstyle \mathfrak{c}}+\underline{\sigma}}{\underline{J}},
\ \ \rho > \tfrac{\Lambda_1}{4 \underline{\sigma}}, \ \ k > \rho \Lambda_1 + \Lambda_2, \label{eq:sigma_under_sigma} \\
\Lambda_1 & \coloneqq \left[L + \overline{J}+ \sigma \left( L_{\scriptscriptstyle \mathfrak{c}} + \widetilde{J}\right) \right], \Lambda_2 \coloneqq \left(L + \sigma \overline{J}^2\right), \label{eq:Lambda_1}
\end{align}
\end{subequations}
renders the sets:
\begin{subequations}
\begin{align}
\Omega_1 & \coloneqq \left\{\mathfrak{e} \in \mathbb{R}^{6} : \|\mathfrak{e}\|_2 \le \tfrac{\widetilde{d}}{\min\{\alpha_1, \alpha_2\}} \right\}, \label{eq:omega_1} \\
\Omega_2 & \coloneqq \left\{\mathfrak{z} \in \mathbb{R}^{6+n} : \|\mathfrak{z}\|_2 \le \tfrac{2 \widetilde{d}}{\overline{J} \min\{\alpha_1, \alpha_2\}} \right\}, \label{eq:omega_2}
\end{align}
\end{subequations}
RCI sets for the error dynamics \eqref{eq:frak_e}, \eqref{eq:frak_z}, according to Definition \ref{def:RPI_set}. The constants $\alpha_1$, $\alpha_2 > 0$ are defined by:
\begin{equation}
\alpha_1 \coloneqq \underline{\sigma}- \tfrac{\Lambda_1}{4 \rho}, \ \alpha_2 \coloneqq k-\rho \Lambda_1-\Lambda_2. \label{eq:a_1_a_2}
\end{equation}
\end{lemma}

\noindent \textbf{Proof :} A backstepping control methodology will be used \cite{krstic1995nonlinear}. The state $\mathfrak{z}$ in \eqref{eq:frak_e} can be seen as virtual input to be designed such that the Lyapunov function $\mathfrak{L}_1(\mathfrak{e}) \coloneqq \frac{1}{2} \|\mathfrak{e}\|^2_2$ for the system \eqref{eq:frak_e} is always decreasing. The time derivative of $\mathfrak{L}_1$ along the trajectories of system \eqref{eq:frak_e} is given by:
\begin{align}
\dot{\mathfrak{L}}(\mathfrak{e}) & = \mathfrak{e}^\top J(\overline{\mathfrak{q}}) \mathfrak{z} + \mathfrak{e}^\top \mathfrak{b}(\cdot) \le \mathfrak{e}^\top J(\overline{\mathfrak{q}}) \mathfrak{z} + L_{\scriptscriptstyle \mathfrak{c}} \|\mathfrak{e}\|_2^2. \label{eq:lyap1}
\end{align}
Design the virtual control input as $\mathfrak{z} \equiv - \sigma J(\overline{\mathfrak{q}})^\top \mathfrak{e}$, with $\underline{J}$, $\sigma$ as given in \eqref{eq:set_Q}, \eqref{eq:sigma_under_sigma}, respectively. Then, by employing \eqref{eq:set_Q}, \eqref{eq:lyap1} becomes:
\begin{align}
\dot{\mathfrak{L}}(\mathfrak{e}) & \le - \sigma \mathfrak{e}^\top J^{+}(\overline{\mathfrak{q}}) \mathfrak{e} + L_{\scriptscriptstyle \mathfrak{c}} \|\mathfrak{e}\|_2^2 \notag \\
& \le - \sigma \lambda_{\min} \left[\tfrac{J^{+}(\overline{\mathfrak{q}})+J^{+}(\overline{\mathfrak{q}})^\top}{2}\right] \|\mathfrak{e}\|_{2}^2 + L_{\scriptscriptstyle \mathfrak{c}} \|\mathfrak{e}\|_2^2 \notag \\
& \le - \sigma \underline{J} \|\mathfrak{e}\|_{2}^2 + L_{\scriptscriptstyle \mathfrak{c}} \|\mathfrak{e}\|_2^2 = - \underline{\sigma} \|\mathfrak{e}\|_{2}^{2}. \label{eq:lyap11}
\end{align}
Define the backstepping auxiliary error state $\mathfrak{r} \coloneqq \mathfrak{z}+\sigma J(\overline{\mathfrak{q}})^\top \mathfrak{e} \in \mathbb{R}^{6+n}$ and the the stack vector $\mathfrak{y} \coloneqq [\mathfrak{e}^\top, \mathfrak{r}^\top]^\top \in \mathbb{R}^{12+n}$. Consider the Lyapunov function $\mathfrak{L}(\mathfrak{y}) = \tfrac{1}{2}\|\mathfrak{y}\|^2$. Its time derivative along the trajectories of the system \eqref{eq:frak_e},\eqref{eq:frak_z} is given by:
\begin{align}
\dot{\mathfrak{L}}(\mathfrak{y}) & = \mathfrak{e}^\top \dot{\mathfrak{e}}+\mathfrak{r}^\top \Big[\dot{\mathfrak{z}}+\sigma J(\overline{\mathfrak{q}})^\top \dot{\mathfrak{e}}+\sigma \dot{J}(\overline{\mathfrak{q}})^\top \mathfrak{e}\Big] \notag \\ 
&\hspace{-5mm} = \left[\mathfrak{e} + \sigma J(\overline{\mathfrak{q}}) \mathfrak{r} \right]^\top \dot{\mathfrak{e}} + \mathfrak{r}^\top \dot{\mathfrak{z}} +\sigma \mathfrak{r}^\top \dot{J}(\overline{\mathfrak{q}})^\top \mathfrak{e} = - \sigma \mathfrak{e}^\top J^{+}(\overline{\mathfrak{q}}) \mathfrak{e}  \notag \\
&\hspace{2mm} + \mathfrak{e}^\top \mathfrak{b}(\cdot) +\sigma \mathfrak{r}^\top J(\overline{\mathfrak{q}})^\top \mathfrak{b}(\cdot) + \mathfrak{e}^\top J(\overline{\mathfrak{q}}) \mathfrak{r} +\sigma \mathfrak{r}^\top J^{+}(\overline{\mathfrak{q}}) \mathfrak{r} \notag \\
&\hspace{2mm}  +\sigma \mathfrak{r}^\top \dot{J}(\overline{\mathfrak{q}})^\top \mathfrak{e} + \mathfrak{r}^\top \mathfrak{l}(\cdot) + \mathfrak{r}^\top (\mathfrak{u}-\overline{u}) + \mathfrak{r}^\top d(\cdot). \label{eq:lyap2}
\end{align}
By invoking \eqref{eq:lyap11} as well as the following:
\begin{align*}
\sigma \mathfrak{r}^\top J(\overline{\mathfrak{q}})^\top \mathfrak{b}(\cdot) & \le \sigma \|\mathfrak{r}\|_{2} \|J(\overline{\mathfrak{q}})\|_{2} \|\mathfrak{b}(\cdot)\|_{2} \le \sigma L_{\scriptscriptstyle \mathfrak{c}} \overline{J} \|\mathfrak{e}\|_{2}  \|\mathfrak{r}\|_{2}, \\
\mathfrak{e}^\top J(\overline{\mathfrak{q}}) \mathfrak{r} & \le \|\mathfrak{e}\|_{2} \|J(\overline{\mathfrak{q}})\|_2 \|\mathfrak{r}\|_{2} \le \overline{J}  \|\mathfrak{e}\|_{2}  \|\mathfrak{r}\|_{2}, \\
\sigma \mathfrak{r}^\top J^{+}(\overline{\mathfrak{q}}) \mathfrak{r} & \le \sigma \|\mathfrak{r}\|_{2}^{2} \|J^{+}(\overline{\mathfrak{q}})\|_{2} \le \sigma \|\mathfrak{r}\|_{2}^{2} \|J(\overline{\mathfrak{q}})\|_{2} \big\|J^{\top}(\overline{\mathfrak{q}}) \big\|_{2} \\
& \le \sigma \overline{J}^2 \|\mathfrak{r}\|_{2}^{2}, \\
\sigma \mathfrak{r}^\top \dot{J}(\overline{\mathfrak{q}})^\top \mathfrak{e} & \le \sigma   \|\mathfrak{e}\|_{2}  \|\dot{J}(\overline{\mathfrak{q}})\|_{2} \|\mathfrak{r}\|_{2} \le \sigma \widetilde{J} \|\mathfrak{e}\|_{2}  \|\mathfrak{r}\|_{2} \\
\mathfrak{r}^\top \mathfrak{l}(\cdot) & \le L \|\mathfrak{e}\|_{2} \|\mathfrak{r}\|_{2} + L \|\mathfrak{r}\|_{2}^{2}, \\
\mathfrak{r}^\top d(\cdot) & \le \|\mathfrak{r}\|_{2} \|d(\cdot)\|_{2} \le \|\mathfrak{y}\|_{2} \widetilde{d},
\end{align*}
\eqref{eq:lyap2} becomes:
\begin{align}
\dot{\mathfrak{L}}(\mathfrak{y}) & \le - \underline{\sigma} \|\mathfrak{e}\|_{2}^{2} + \Lambda_1 \|\mathfrak{e}\|_{2}  \|\mathfrak{r}\|_{2} \notag  \\
&\hspace{16mm} + \Lambda_2 \|\mathfrak{r}\|^{2}_{2} + \mathfrak{r}^\top (\mathfrak{u}-\overline{u}) + \|\mathfrak{y}\|_{2} \widetilde{d}. \label{eq:lyap3}
\end{align}
with $\Lambda_1$, $\Lambda_2$ given in \eqref{eq:Lambda_1}. By using Lemma \ref{lemma:basic_ineq} for $n = P = 1$, we get $\|\mathfrak{e}\|_{2} \|\mathfrak{r}\|_{2}$ $\le \tfrac{1}{4 \rho} \|\mathfrak{e}\|_{2}^{2}$ $+ \rho \|\mathfrak{r}\|_{2}^{2}$,
with $\rho$ designed so that \eqref{eq:sigma_under_sigma} holds. Combining the latter with \eqref{eq:lyap3} it yields:
\begin{align*}
\dot{\mathfrak{L}}(\mathfrak{y}) & \le - \left(\underline{\sigma}- \tfrac{\Lambda_1}{4 \rho} \right) \|\mathfrak{e}\|_{2}^{2}  + \big(\rho \Lambda_1 + \Lambda_2 \big) \|\mathfrak{r}\|^{2}_{2} \notag \\
&\hspace{30mm} + \mathfrak{r}^\top (\mathfrak{u}-\overline{u}) + \|\mathfrak{y}\|_{2} \widetilde{d}.
\end{align*}
By designing $\mathfrak{u} - \overline{u} = -k \mathfrak{r} = -k \mathfrak{e}-k \sigma J(\overline{\mathfrak{q}})^{\top} \mathfrak{z}$, which is compatible with \eqref{eq:control_law_u} and the same as in \eqref{eq:kappa_law}, we have:
\begin{align*}
\dot{\mathfrak{L}}(\mathfrak{y}) & \le - \left(\underline{\sigma}- \tfrac{\Lambda_1}{4 \rho} \right) \|\mathfrak{e}\|_{2}^{2}  - \big(k -\rho \Lambda_1 - \Lambda_2 \big) \|\mathfrak{r}\|^{2}_{2} + \|\mathfrak{y}\|_{2} \widetilde{d} \\
& \le - \min\{\alpha_1, \alpha_2\}\|\mathfrak{y}\|_2^2 + \|\mathfrak{y}\|_2 \widetilde{d} \\
& = -\|\mathfrak{y}\|_2 \big[ \min\{\alpha_1, \alpha_2\}\|\mathfrak{y}\|_2 - \widetilde{d} \big],
\end{align*}
as $\alpha_1$ and $\alpha_2$ given in \eqref{eq:a_1_a_2}. Thus, $\dot{\mathfrak{L}}(\mathfrak{y}) < 0$, when $\|\mathfrak{y}\|_2 > \tfrac{\widetilde{d}}{\min\{\alpha_1, \alpha_2\}}$. Taking the latter into consideration and the fact that $\mathfrak{y}(0)$, we have that $\|\mathfrak{y}(t)\| \le \tfrac{\widetilde{d}}{\min\{\alpha_1, \alpha_2\}}$, $\forall t \ge 0$. Moreover, the following inequalities hold:
\begin{align*}
\|\mathfrak{e}\|_2  &\le \| \mathfrak{y} \|_2 \Rightarrow \|\mathfrak{e}(t)\|_2 \le \tfrac{\widetilde{d}}{\min\{\alpha_1, \alpha_2\}}, \forall t \ge 0, \\
\Big| \|\mathfrak{e}\|_2- \big\|J^\top \mathfrak{z} \big\|_2 \Big| &\le \big\|\mathfrak{e}+J^\top \mathfrak{z} \big\|_2  = \|\mathfrak{z}\|_2 \le \|\mathfrak{y}\|_2 \\
\Rightarrow \|\mathfrak{z}(t)\|_2 &\le \tfrac{2 \widetilde{d}}{\overline{J} \min\{\alpha_1, \alpha_2\}}, \forall t \ge 0. \hspace{34mm} \square
\end{align*}

\begin{remark}
According to Lemma \ref{lemma:tube}, the volume of the tube which is centered along the nominal trajectories $\overline{e}(t)$, $\overline{\zeta}(t)$, that are solution of system \eqref{eq:nom_error_kin},\eqref{eq:nom_error_dyn}, depends on the parameters $\widetilde{d}$, $\overline{J}$, $\underline{J}$, $\widetilde{J}$, $L$ and $L_{\scriptscriptstyle \mathfrak c}$. By tuning the parameters $\rho$ and $k$ from \eqref{eq:sigma_under_sigma} appropriately, the volume of the tube can be adjusted.
\end{remark}

\subsection{On-line Optimal Control} \label{sec:optimal_contol}

Consider a sequence of sampling times $\{t_k\}$, $k \in \mathbb{N}$, with a constant sampling period $0 < h < T$, where $T$ is a prediction horizon such that $t_{k+1} \coloneqq t_{k} + h$, $\forall k \in \mathbb{N}$. At each sampling time $t_k$, a FHOCP is solved as follows:
\begin{subequations}
\begin{align}
&\hspace{-7mm}\min\limits_{\overline{u}(\cdot)} \left\{  \|\overline{\xi}(t_k+T)\|^2_{\scriptscriptstyle P} \hspace{-1mm} + \hspace{-2mm}\int_{t_k}^{t_k+T} \hspace{-1mm}\Big[ \|\overline{\xi}(\mathfrak{s})\|^2_{\scriptscriptstyle Q} +\|\overline{u}(\mathfrak{s})\|^2_{\scriptscriptstyle R} \Big] d\mathfrak{s} \right\} \hspace{0mm} \label{eq:mpc_cost_function} \hspace{-7mm}\\
&\hspace{-6mm}\text{subject to:} \notag \\
&\hspace{-3mm} \dot{\overline{\xi}}(\mathfrak{s}) = g(\overline{\xi}(\mathfrak{s}), \overline{u}(\mathfrak{s})), \ \ \overline{\xi}(t_k) = \xi(t_k), \label{eq:diff_mpc} \\
&\hspace{-3mm} \overline{\xi}(\mathfrak{s}) \in \overline{\mathcal{E}} \times \overline{\mathcal{Z}}, \ \ \overline{u}(\mathfrak{s}) \in \overline{\mathcal{U}},  \ \ \forall \mathfrak{s} \in [t_k,t_k+T], \label{eq:mpc_constrained_set} \\
&\hspace{-3mm} \overline{\xi}(t_k+T)\in \mathcal{F}, \label{eq:mpc_terminal_set}
\end{align}
\end{subequations}
where $\xi \hspace{-1mm}\coloneqq\hspace{-1mm}[e^\top,\zeta^\top]^\top \hspace{-2mm}\in \mathbb{R}^{12+n}$, $g(\xi,u)\hspace{-1mm}$ $\coloneqq\hspace{-1mm}\begin{bmatrix} J(\mathfrak{q}) \zeta \\ f(e+\chi_{\scriptscriptstyle \rm des}, \zeta)+u \end{bmatrix}$; $Q$, $P \in \mathbb{R}^{(12+n) \times (12+n)}$ and $R \in \mathbb{R}^{(6+n) \times (6+n)}$ are positive definite gain matrices to be appropriately tuned. We will explain hereafter the sets $\overline{\mathcal{E}}$, $\overline{\mathcal{V}}$, $\overline{\mathcal{U}}$ and $\mathcal{F}$.

In order to guarantee that while the FHOCP \eqref{eq:mpc_cost_function}-\eqref{eq:mpc_terminal_set} is solved for the nominal dynamics \eqref{eq:nom_error_kin}-\eqref{eq:nom_error_dyn}, the real states $e$, $\zeta$ and control input $\mathfrak{u}$ satisfy the corresponding state $\mathcal{E}$, $\mathcal{Z}$ and input constraints $\mathcal{U}$, respectively, the following modification is performed: $\overline{\mathcal{E}} \coloneqq \mathcal{E} \ominus \Omega_1, \ \ \overline{\mathcal{Z}} \coloneqq \mathcal{Z} \ominus \Omega_2, \ \ \overline{\mathcal{U}} \coloneqq \mathcal{U} \ominus \left[ \Lambda \circ \overline{\Omega} \right]$,
with $\Lambda \coloneqq {\rm diag} \{-k I_6, -k \sigma \overline{J} I_{6+n}\} \in \mathbb{R}^{(12+n) \times (12+n)}$, $\overline{\Omega} \coloneqq \Omega_1 \times \Omega_2$, the operators $\ominus$, $\circ$ as defined in Section \ref{sec:notation_preliminaries}, and $\Omega_1$, $\Omega_2$ as given in \eqref{eq:omega_1}, \eqref{eq:omega_2}, respectively. Intuitively, the sets $\mathcal{E}$, $\mathcal{Z}$ and $\mathcal{U}$ are tightened accordingly, in order to guarantee that while the nominal states $\overline{e}$, $\overline{\zeta}$ and the nominal control input $\overline{u}$ are calculated, the corresponding real states $e$, $\zeta$ and real control input $\mathfrak{u}$ satisfy the state and input constraints $\mathcal{E}$, $\mathcal{Z}$ and $\mathcal{U}$, respectively. This constitutes a standard constraints set modification technique adopted in tube-based NMPC frameworks (for more details see \cite{yu_2013_tube}). Define the \emph{terminal set} by:
\begin{align} \label{eq:terminal_set_F}
\mathcal{F} \coloneqq \big\{\overline{\xi} \in \overline{\mathcal{E}} \times \overline{\mathcal{Z}} : \|\overline{\xi}\|_{\scriptscriptstyle P} \le \epsilon \big\}, \ \  \epsilon > 0,
\end{align}
which is used to enforce the stability of the system \cite{frank_1998_quasi_infinite}. In particular, due to the fact that the linearized nominal dynamics $\dot{\overline{\xi}} = A \overline{\xi} + B \overline{u}$ are stabilizable, it can be proven that (see \cite[Lemma 1, p. 4]{frank_1998_quasi_infinite}) there exists a \emph{local controller} $u_{\scriptscriptstyle \rm loc} \coloneqq \mathfrak{K} \overline{\xi} \in \overline{\mathcal{U}}$, $\mathfrak{K} \in \mathbb{R}^{(6+n) \times (6+n)}$, $\mathfrak{K} > 0$ which guarantees that: $\tfrac{d}{dt}\left(\|\overline{\xi}\|^2_{\scriptscriptstyle P}\right) \le -\|\overline{\xi}\|^2_{\scriptscriptstyle \widetilde{Q}}$, $\forall \overline{\xi} \in \mathcal{F}$, with $\widetilde{Q} \coloneqq Q+\mathfrak{K}^\top R$.

\begin{theorem} \label{theorem_main}
Suppose also that the FHOCP \eqref{eq:mpc_cost_function}-\eqref{eq:mpc_terminal_set} is feasible at time $t = 0$. Then, the feedback control law \eqref{eq:control_law_u} applied to the system \eqref{eq:unsrt_error_kin}-\eqref{eq:unsrt_error_dyn} guarantees that there exists a time $\mathfrak{t}$ such that $\forall t \ge \mathfrak{t}$ it holds that:
\begin{subequations}
\begin{align}
\hspace{0mm} \|\chi(t)-\chi_{\scriptstyle \rm des}\|_{\scriptscriptstyle 2} & \le \tfrac{\epsilon}{\sqrt{\lambda_{\scriptscriptstyle \min}(P)}} + \tfrac{\widetilde{d}}{\min\{\alpha_1, \alpha_2\}}, \label{eq:theom_ineq_1} \\
\hspace{0mm} \|\zeta(t)\|_{\scriptscriptstyle 2} & \le \tfrac{\epsilon}{\sqrt{\lambda_{\scriptscriptstyle \min}(P)}} + \tfrac{2 \widetilde{d}}{\overline{J} \min\{\alpha_1, \alpha_2\}}. \label{eq:theom_ineq_2}
\end{align}
\end{subequations}
\end{theorem}
\begin{proof}
The proof of the theorem consists of two parts:
	
\noindent \textbf{Feasibility Analysis}: It can be shown that recursive feasibility is established and it implies subsequent feasibility. The proof of this part is similar to the feasibility proof of \cite[Theorem 2, Sec. 4, p. 12]{alex_IJRNC_2018}, and it is omitted here due to space constraints.
	
\noindent \textbf{Convergence Analysis}: Recall that $e = \chi-\chi_{\scriptscriptstyle \rm des}$, $\mathfrak{e} = e-\overline{e}$ and $\mathfrak{z} = \zeta-\overline{\zeta}$. Then, we get $\|\chi(t)-\chi_{\scriptstyle \rm des}\|_{\scriptscriptstyle 2}$  $\le \|\overline{e}(t)\|_{\scriptscriptstyle 2} + \|\mathfrak{e}(t)\|_{\scriptscriptstyle 2}$, $\|\zeta(t)\|_{\scriptscriptstyle 2}$ $\le \|\overline{\zeta}(t)\|_{\scriptscriptstyle 2} + \|\mathfrak{z}(t)\|_{\scriptscriptstyle 2}$, which, by using the fact that $\|\overline{e}\|$, $\|\overline{\zeta}\| \le \|\overline{\xi}\|_{2}$ as well as the bounds from \eqref{eq:omega_1}, \eqref{eq:omega_2} the latter inequalities become:
\begin{subequations}
\begin{align}
\|\chi(t)-\chi_{\scriptstyle \rm des}\|_{\scriptscriptstyle 2} & \le \|\overline{\xi}(t)\|_{\scriptscriptstyle 2} +  \tfrac{ \widetilde{d}}{\min\{\alpha_1, \alpha_2\}}, \label{eq:conv_1}\\
\|\zeta(t)\|_{\scriptscriptstyle 2} & \le \|\overline{\xi}(t)\|_{\scriptscriptstyle 2} +  \tfrac{2 \widetilde{d}}{\overline{J} \min\{\alpha_1, \alpha_2\}}, \forall t \ge 0. \label{eq:conv_2}
\end{align}
\end{subequations}
The nominal state $\overline{\xi}$ is controlled by the nominal control action $\overline{u} \in \overline{\mathcal{U}}$ which is the outcome of the solution to the  FHOCP \eqref{eq:mpc_cost_function}-\eqref{eq:mpc_terminal_set} for the nominal dynamics \eqref{eq:nom_error_kin}-\eqref{eq:nom_error_dyn}. Hence, by invoking previous NMPC stability results found in \cite{frank_1998_quasi_infinite}, the state $\overline{\xi}(t)$ is driven to terminal set $\mathcal{F}$, given in \eqref{eq:terminal_set_F}, in finite time, and it remains there for all times. Thus, there exist a finite time $\mathfrak{t}$ such that $\overline{\xi}(t) \in \mathcal{F}$, $\forall t \ge \mathfrak{t}$. From \eqref{eq:terminal_set_F}, the latter implies that: $\|\overline{\xi}(t)\|_{\scriptscriptstyle P} \le \epsilon, \forall t \ge \mathfrak{t} \Rightarrow \|\overline{\xi}(t)\|_{\scriptscriptstyle 2} \le \tfrac{\epsilon}{\sqrt{\lambda_{\scriptscriptstyle \min}(P)}}, \forall t \ge \mathfrak{t}.$ The latter implication combined by \eqref{eq:conv_1}-\eqref{eq:conv_2} leads to the conclusion of the proof.
\end{proof}

\begin{figure}[t!]
	\centering
	\includegraphics[scale = 0.4]{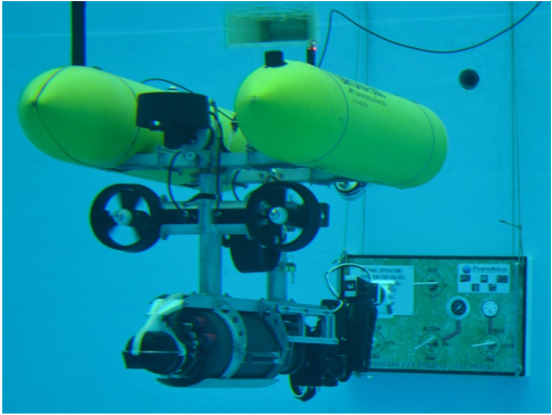}
	\caption{ The GIRONA-UVMS composed of Girona500 AUV and ARM 5E Micro manipulator \cite{cieslak2015autonomous}.}\label{fig:girona}
\end{figure}
\section{Simulation Results} \label{sec:simulation_results}
For a simulation scenario, consider the Girona 500 AUV depicted in Fig. \ref{fig:girona} equipped with an ARM 5E Micro manipulator from \cite{cieslak2015autonomous}. The manipulator consists of $n = 4$ revolute joints with limits: $-0.52 \le q_1 \le 1.46$, $\-0.1471 \le q_2 \le 1.3114$, $-1.297 \le q_3 \le 0.73$ and $-3.14 \le q_4 \le 3.14$. The end-effector is in ready-to-grasp mode with initial state: $\chi(0) = [\mathfrak{p}(0)^\top, \mathfrak{o}(0)^\top]^\top =[-1.0, 1.3, -1.0, 0.0, -\tfrac{\pi}{8}, \tfrac{\pi}{12}]^\top$. The stiffness matrix is $K = I_{6}$ with $\chi_{\scriptscriptstyle \rm eq} = 0$ which results to $\mathfrak{F}_{\scriptscriptstyle \rm des} = \chi_{\scriptscriptstyle \rm des} = [\mathfrak{p}_{\scriptscriptstyle \rm des}^\top, \mathfrak{o}_{\scriptscriptstyle \rm des}^\top ]^\top = [0, 0, 0, \tfrac{\pi}{3}, \tfrac{\pi}{10}, 0]^\top$.
According to \eqref{eq:forw_kinematics}, the transformation matrices which lead to the forward kinematics are given by:
\begin{align*}
T_B^I & = 
\begin{bmatrix}
\mathfrak{J}_1(\eta_2) & \eta_1 \\
0_{1 \times 3} & 1 \\
\end{bmatrix}, \ \ T_0^B = 
\begin{bmatrix}
I_{3 \times 3} & \left[ 0.53, 0, 0.36 \right]^\top \\
0_{1 \times 3} & 1 \\
\end{bmatrix},
\end{align*}
and $T_{i}^{i-1}$, $i =1,\dots,4$ are given by the Denavit-Hantenberg parameters which can be calculated from Table \ref{table:DH_parameters}. By imposing the constraints $-\pi \le \phi$, $\psi \le \pi$ and $-\tfrac{\pi}{2}+\epsilon \le \theta \le \tfrac{\pi}{2} - \epsilon$, $\epsilon = 0.1$, according to \eqref{eq:set_Q} we get $\underline{J} = 0.5095$ and $L_{\scriptscriptstyle \mathfrak{c}} = 2 \sqrt{2}$. For simplified calculations, we apply the methodology of this paper by considering disturbance in the following disturbed kinematic model: $\dot{\chi} = J(\mathfrak{q}) \zeta + w(\mathfrak{q}, t)$,
with $w(\cdot) = 0.2 \sin(t) I_6$ $\Rightarrow \|w(\cdot)\|_{2} \le 0.2 = \widetilde{w}$, in which the vector $\zeta$ stands for the virtual control input to be designed such that $\lim_{t \to \infty} \|\chi(t) - \chi_{\scriptscriptstyle \rm des}\| \to 0$. The input constraints are set to $\|\nu_1\|_2 \le 2$, $\|\nu_2\|_2 \le 2$ and $\|\dot{q}\|_2 \le 2$. Then, by using \eqref{eq:frak_e} and \eqref{eq:lyap1} and designing the control gain $\sigma = 3.084$, the resulting RCI is $\Omega = \left\{\mathfrak{e} \in \mathbb{R}^{6} : \|\mathfrak{e}\|_{2} \le \tfrac{\widetilde{w}}{\sigma \underline{J} + L_{\scriptscriptstyle \mathfrak{c}}} = 0.3 \right\}$.
\begin{table}[t!]
\begin{center}
\begin{tabular}{|C{0.4cm}||C{1.2cm}||C{1.2cm}||C{1.2cm}||C{1.2cm}|}
\hline
& $d_i (m)$ & $q_i$ & $a_i (m)$ & $\alpha_i (\text{rad})$ \\
\hline \hline
$1$  & $0$ & $q_1$ & $0.1$ & $-\frac{\pi}{2}$ \\
\hline
$2$ & $0$ & $q_2$ & $0.26$ & $0$ \\
\hline
$3$ & $0$ & $q_3$ & $0.09$ & $\frac{\pi}{2}$  \\
\hline
$4$ & $0.29$ & $q_4$ & $0$ & $0$ \\
\hline \hline
E &  \multicolumn{4}{c|}{$\text{Rot}(y,-\frac{\pi}{2})$} \\
\hline
\end{tabular}
\end{center}
\caption{Denavit-Hantenberg Parameters of the ARM 5E Micro}
\label{table:DH_parameters}
\end{table}
The simulation time is $6 \sec$. The optimization horizon and the sampling time are set to $T = 0.7 \sec$ and $h = 0.1 \sec$, respectively. The NMPC gains are set to $Q = P = 0.5 I_{6}$ and $R = 0.5 I_{10}$. Fig. \ref{fig:error} shows the evolution of the real and the nominal position errors of the end-effector. the corresponding real and nominal orientation errors are depicted in Fig. \ref{fig:error2}. Finally, the control inputs are presented in Fig. \ref{fig:inputs}. It can be observed that the desired task is performed while all the state/input constraints are satisfied.

%The simulation was conduced in MATLAB R2015a by using optimization tools found in \cite{grune2016nonlinear}. It takes $34.2 \sec$ on a laptop with $4$ cores, i7-$2.80$ GHz CPU and~$16$GB~of~RAM. 

\section{Conclusions and Future Research} \label{sec:conclusions}

This paper addresses the problem of force/torque control of UVMS under state/input constraints as well as external uncertainties/disturbances. In particular, we have proposed a tube-based robust NMPC framework that incorporates the aforementioned constraints in a novel way. Future efforts will be devoted towards extending the current framework under multi-UVMS which interact with each other through a common object in order to perform a collaborative manipulation task. 

\begin{figure}[t!]
\centering
\includegraphics[scale = 0.45]{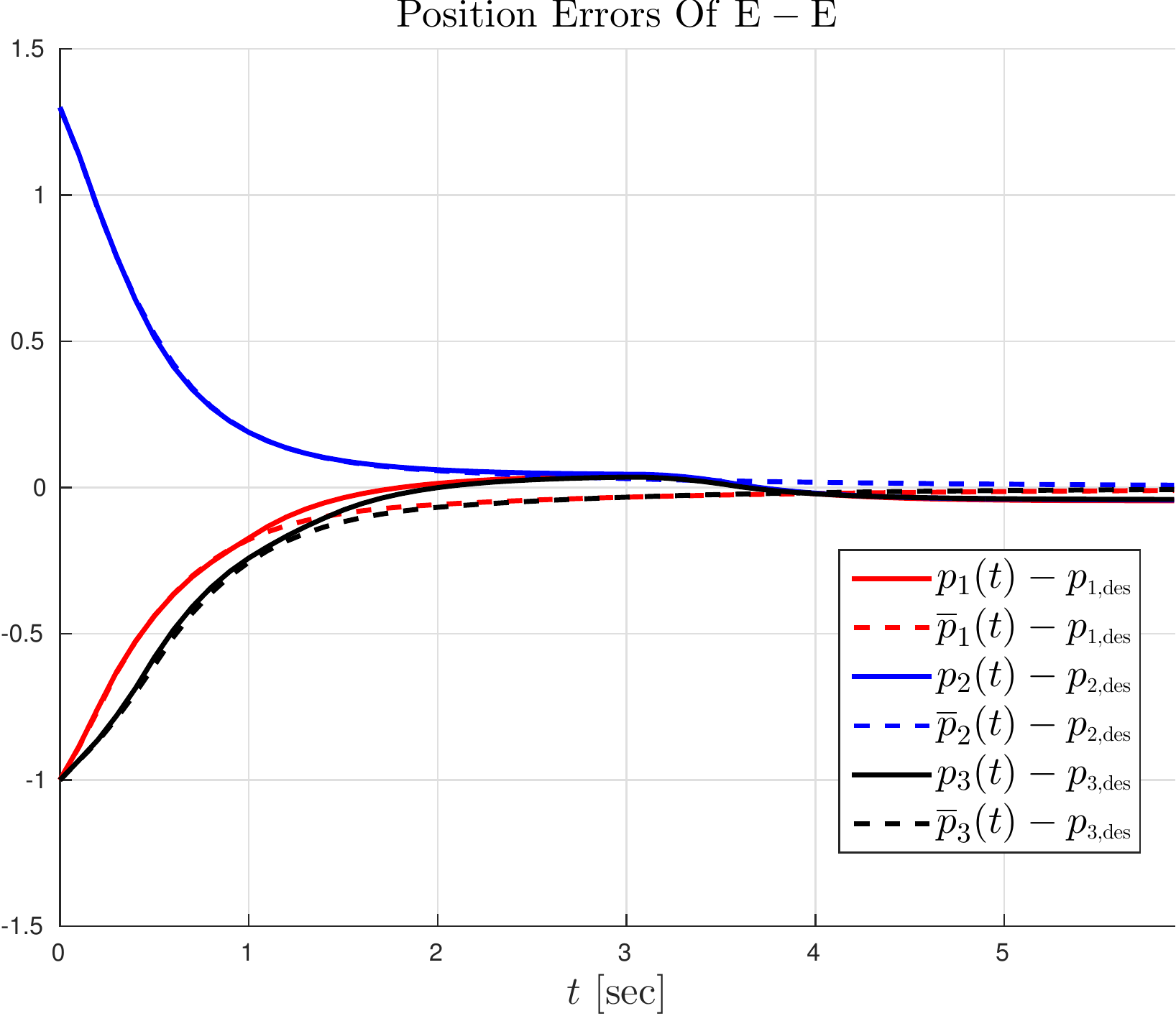}
\caption{The evolution of the real position errors of the end-effector $\mathfrak{p}_1(t)-\mathfrak{p}_{\scriptscriptstyle \rm 1,des}$, $\mathfrak{p}_2(t)-\mathfrak{p}_{\scriptscriptstyle \rm 2,des}$, $\mathfrak{p}_3(t)-\mathfrak{p}_{\scriptscriptstyle \rm 3,des}$ depicted with solid lines as well as the corresponding nominal position errors $\overline{\mathfrak{p}}_1(t)-\mathfrak{p}_{\scriptscriptstyle \rm 1,des}$, $\overline{\mathfrak{p}}_2(t)-\mathfrak{p}_{\scriptscriptstyle \rm 2,des}$, $\overline{\mathfrak{p}}_3(t)-\mathfrak{p}_{\scriptscriptstyle \rm 3,des}$ depicted with dashed lines.}\label{fig:error}
\end{figure}

\begin{figure}[t!]
\centering
\includegraphics[scale = 0.45]{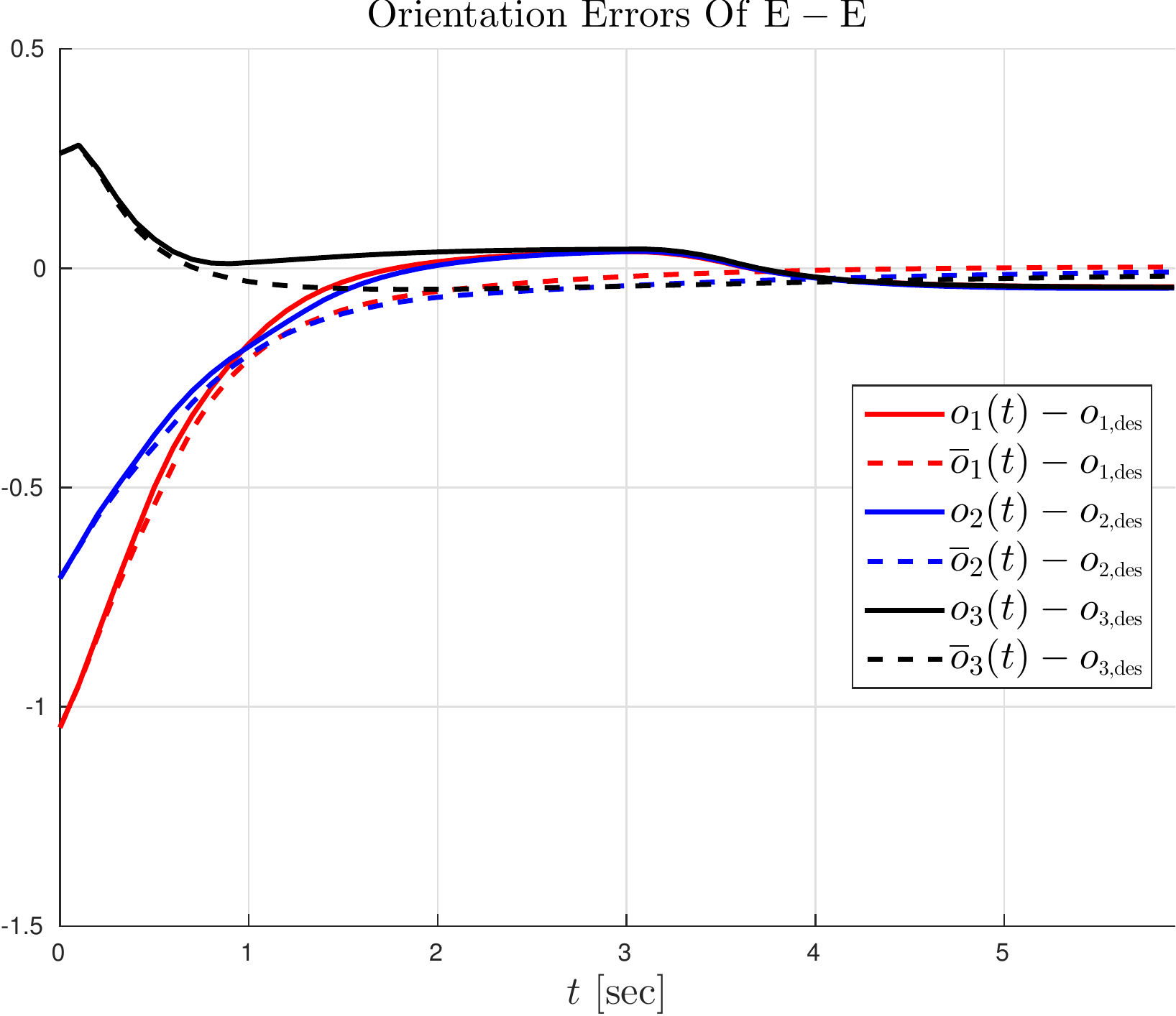}
\caption{The evolution of the real orientation errors $\mathfrak{o}_1(t)-\mathfrak{o}_{\scriptscriptstyle \rm 1,des}$, $\mathfrak{o}_2(t)-\mathfrak{o}_{\scriptscriptstyle \rm 2,des}$, $\mathfrak{o}_3(t)-\mathfrak{o}_{\scriptscriptstyle \rm 3,des}$ depicted with solid lines as well as the corresponding nominal orientation errors $\overline{\mathfrak{o}}_1(t)-\mathfrak{o}_{\scriptscriptstyle \rm 1,des}$, $\overline{\mathfrak{o}}_2(t)-\mathfrak{o}_{\scriptscriptstyle \rm 2,des}$, $\overline{\mathfrak{o}}_3(t)-\mathfrak{o}_{\scriptscriptstyle \rm 3,des}$ depicted with dashed lines.}\label{fig:error2}
\end{figure}
\begin{figure}[t!]
\centering
\includegraphics[scale = 0.45]{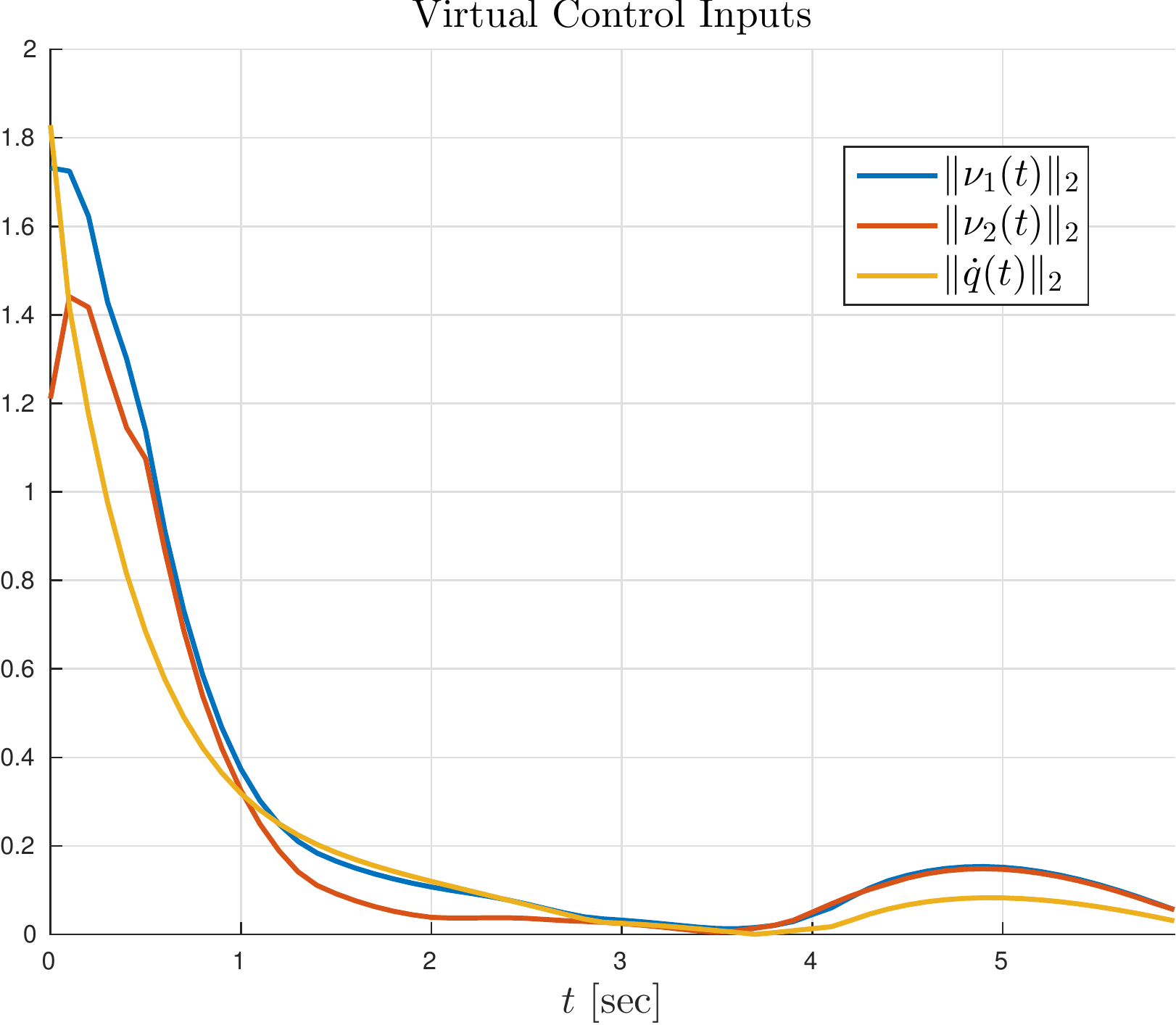}
\caption{The virtual control input signals $\|\nu_1(t)\|_{2}$, $\|\nu_2(t)\|_{2}$ and $\|\dot{q}(t)\|_2$ of the kinematic model \eqref{eq:kinematics}. It holds that $\|\nu_1\|_2 \le 2$, $\|\nu_2\|_2 \le 2$ and $\|\dot{q}\|_2 \le 2$.}\label{fig:inputs}
\end{figure}

\bibliographystyle{ieeetr}
\bibliography{references}
\end{document}